\documentclass[a4paper,11pt,reqno]{article}

\usepackage[utf8]{inputenc}
\usepackage[T1]{fontenc}
\usepackage{lmodern}
\usepackage[english]{babel}
\usepackage{microtype}

\usepackage{amsmath,amssymb,amsfonts,amsthm}
\usepackage{mathtools,accents}
\usepackage{mathrsfs}
\usepackage{aliascnt}
\usepackage{braket}
\usepackage{bm}

\usepackage[a4paper,margin=3cm]{geometry}
\usepackage[citecolor=blue,colorlinks]{hyperref}

\usepackage{enumerate}
\usepackage{xcolor}

\makeatletter
\g@addto@macro\@floatboxreset\centering
\makeatother

\makeatletter
\def\newaliasedtheorem#1[#2]#3{
  \newaliascnt{#1@alt}{#2}
  \newtheorem{#1}[#1@alt]{#3}
  \expandafter\newcommand\csname #1@altname\endcsname{#3}
}
\makeatother

\numberwithin{equation}{section}

\newtheoremstyle{slanted}{\topsep}{\topsep}{\slshape}{}{\bfseries}{.}{.5em}{}

\theoremstyle{plain}
\newtheorem{theorem}{Theorem}[section]
\newaliasedtheorem{proposition}[theorem]{Proposition}
\newaliasedtheorem{lemma}[theorem]{Lemma}
\newaliasedtheorem{corollary}[theorem]{Corollary}

\theoremstyle{definition}
\newaliasedtheorem{definition}[theorem]{Definition}

\theoremstyle{remark}
\newaliasedtheorem{remark}[theorem]{Remark}

\newcommand{\diam}{\text{diam}}
\newcommand{\dist}{\text{dist}}

\begin{document}
\title{Some remarks on the Gehring-Hayman theorem}

\author{Sari Rogovin \thanks{Hakarinteen koulu, Tammisaari, Finland, \url{sari.rogovin@gmail.com}.}, Hyogo Shibahara
\thanks{Department of Mathematical Sciences, University of Cincinnati, Cincinnati, OH 45221, U.S.A., \url{shibahho@mail.uc.edu}.} and Qingshan Zhou \thanks{School of Mathematics and Big Data, Foshan University,  Foshan, Guangdong 528000, People's Republic of China, \url{qszhou1989@163.com; q476308142@qq.com}.}} 
\maketitle

\begin{abstract}
In this paper we provide new characterizations of the Gehring-Hayman theorem from the point of view of Gromov boundary and uniformity. We also determine the critical exponents for the uniformized space to be a uniform space in the case of the hyperbolic spaces, the model spaces $\mathbb{M}^{\kappa}_n$ of the sectional curvature $\kappa<0$ with the dimension $n \geq 2$ and hyperbolic fillings.
\end{abstract}


\allowdisplaybreaks{

\section{Introduction}
The aim of this paper is to provide some remarks on the uniformizations of $\delta$-Gromov hyperbolic spaces developed by Bonk, Heinonen and Koskela in \cite{BHK}. For a $\delta$-Gromov hyperbolic space $(X, d)$ with a distinguished point $p \in X$, they introduced the metric defined by  
\begin{equation}\label{uniformization metric}
	d_{\epsilon}(x, y):= \inf\limits_{\gamma}\int_{\gamma}\rho_{\epsilon} \,ds,
\end{equation}
where $\rho_{\epsilon}(\cdot)=e^{-\epsilon d(p, \ \cdot \ )}$ and the infimum is taken over all rectifiable curves $\gamma$ from $x$ to $y$. The metric space $(X, d_{\epsilon})$ is called \emph{a uniformized space of $X$} and is denoted simply by $X^{\epsilon}$. This uniformization technique has led to numerous applications, see \cite{BB}, \cite{BHK} and \cite{BBS} and references therein. Another uniformization procedure using Busemann functions has been also established, see \cite{B} and \cite{Z}.

 It was shown in \cite[Proposition 4.5]{BHK} that if a $\delta$-Gromov hyperbolic space $(X, d)$ is uniformized with a sufficiently small parameter $\epsilon>0$, then $X^{\epsilon}$ is a uniform space. The key to prove the above result was to show the following Gehring-Hayman theorem.
\begin{theorem}(\cite[Section 5]{BHK}\label{GH theorem})
	Let $(X, d)$ be a $\delta$-Gromov hyperbolic space. Then there exists $\epsilon(\delta)>0$ such that for every $0<\epsilon \leq \epsilon(\delta)$, there exists $A>0$ such that for each pair of points $x, y \in X^{\epsilon}$,
	\begin{equation}\label{GH theorem}
		l_{d_{\epsilon}}([x, y]) \leq A d_{\epsilon}(x, y),
	\end{equation}
	where $[x, y]$ is a geodesic curve with respect to $d$ from $x$ to $y$ and $l_{d_{\epsilon}}([x, y])$ denotes the length of $[x, y]$ with respect to $d_{\epsilon}$. 
\end{theorem}
 In this paper we say that \emph{the Gehring-Hayman theorem holds for $X^{\epsilon}$} if \eqref{GH theorem} holds for all $x, y \in X^{\epsilon}$. We prove that $X^{\epsilon}$ being uniform implies the Gehring-Hayman theorem for $X^{\epsilon}$ through quasihyperbolization, see Subsection \ref{From uniformity to GH}. We also study the following two localized Gehring-Hayman properties.
\begin{definition}[Localized Gehring-Hayman property]\label{GH for Gromov sequences}
	We say that the \emph{Gehring-Hayman property for Gromov sequences} holds if for each Gromov sequence $(x_n)_n$, there exists $C \geq 1$ such that for all $n, m \in \mathbb{N}$
	\begin{equation}\label{GH for sequences}
		l_{d_{\epsilon}}([x_n, x_m]) \leq C d_{\epsilon}(x_n, x_m).
	\end{equation}
	Note that the constant $C$ may depend on $(x_n)_n$. See Remark \ref{equivalence relation for Gromov sequences} for the equivalence relation among Gromov sequences. We also say that  the \emph{Gehring-Hayman property for metric boundary points} holds if for all $(x_n)_n$ and $x \in \partial_{d_{\epsilon}} X^{\epsilon}$ such that $d_{\epsilon}(x_n, x) \to 0$ as $n \to  \infty$, there exists $C \geq 1$ such that \eqref{GH for sequences} holds for all $n, m \in \mathbb{N}$.
\end{definition}
It is clear that the original Gehring-Hayman theorem implies the localized Gehring-Hayman properties and the bijectivity of the canonical boundary map $\Phi : \partial_{G} X \to \partial_{d_{\epsilon}} X^{\epsilon}$, see \cite[Proposition 4.13]{BHK} and also Remark \ref{Def of boundary map} for the construction of the map $\Phi$. We obtain further characterizations of  the Gehring-Hayman theorem from these properties. The first two conditions in Theorem \ref{GH decomposition 1} can be seen as a boundary pointwise decomposition of the Gehring-Hayman theorem. The following is the list of equivalent conditions of the Gehring-Hayman theorem.
\begin{theorem}\label{GH decomposition 1}
	Let $(X, d)$ be an $M$-roughly starlike $\delta$-Gromov hyperbolic space and $X^{\epsilon}$ be the uniformized space with $\epsilon>0$. Then the following are equivalent.
	\begin{enumerate}
	\item The Gehring-Hayman property for metric boundary points holds.
		\item The canonical boundary map $\Phi :  \partial_{G} X \to \partial_{d_{\epsilon}}X^{\epsilon}$ is  bijective and the Gehring-Hayman property for Gromov sequences holds.
		\item The Gehring-Hayman theorem holds.
		\item $X^{\epsilon}$ is a uniform space.
	\end{enumerate}
\end{theorem}


We remark that the equivalence among the first three conditions in Theorem \ref{GH decomposition 1} holds true without the roughly starlike property. An immediate consequence of Theorem \ref{GH decomposition 1} together with \cite[Proposition 4.12]{B} is the following.
\begin{corollary}
	Let $(X, d)$ be an $M$-roughly starlike $\delta$-Gromov hyperbolic space and $X^{\epsilon}$ be the uniformized space with $\epsilon>0$. Suppose that $X^{\epsilon}$ is a uniform space. Then $X^{\epsilon'}$ is a uniform space for all $0<\epsilon'\leq \epsilon$.
\end{corollary}

Due to Theorem \ref{GH decomposition 1}, looking at boundary behavior makes it easy to check that $X^{\epsilon}$ is not a uniform space. Using the results in \cite{B} and \cite{BBS}, we determine the sharp uniformization parameter for the uniformized space to be a uniform space in the case of the hyperbolic spaces, the model spaces $\mathbb{M}^{\kappa}_n$ of the sectional curvature $\kappa<0$ with the dimension $n \geq 2$ and hyperbolic fillings, see Section~\ref{examples}. We note that in the case of the hyperbolic spaces,  Butler has already shown the same result by a different argument, see \cite[Remark 1.~11]{B}. He employed the results of the asymptotic curvature upper bound \cite[Definition~1.1 and Theorem~1.5]{BF} and the fact that the Gehring-Hayman theorem induces a visual metric on the Gromov boundary $\partial_{G}X$. Our argument is more elementary in that we only use the polar coordinate expression of the hyperbolic metric.

	

 \subsection*{Acknowledgment}
The authors thank Prof. Nageswari Shanmugalingam for introducing the topic and discussions on it. Qingshan Zhou was partly supported by NNSF of China (Nos. 11901090 and 12071121), by Department of Education of Guangdong Province, China (No. 2021KTSCX116), by Guangdong Basic and Applied Basic Research Foundation (No. 2021A1515012289).

\section{Preliminaries}
We fix some notation. Let $(X, d)$ be a metric space. We say that $(X, d)$ is proper if all bounded closed sets are compact. For $x \in X$ and $r>0$, we set $B_d(x,r):=~\{y\in X\, |\, d(x,y)<r\}$. We also set $\dist_{d}(x, A):=~\inf_{y \in A}d(x, y)$ for any $A \subseteq X$ and $x \in X$. The minimum of $n$ real numbers $(a_i)_{i=1}^{n}$ is denoted by $a_1\wedge \cdots \wedge a_n$.  The \emph{length} of a curve $\gamma$ is denoted by $l_d(\gamma)$. A curve $\gamma$ is said to be \emph{rectifiable} if $l_d(\gamma) <\infty$. For $a, b\in~(-\infty, \infty)$, we say that a curve $\gamma : [a, b] \to X$ is \emph{geodesic} if $d(\gamma(a), \gamma(b))=l_d(\gamma)$. A geodesic curve from $x$ to $y$ is often denoted by $[x, y]$. A curve $\gamma : [0, \infty) \to X$ is called a \emph{geodesic ray} if $\gamma|_{[0, t]}$ is geodesic for every $t >0$. We first introduce the $M$-roughly starlike property.
\begin{definition}[$M$-Roughly starlike property]
	Let $M\geq 0$. We say that a metric space $(X, d)$ is \emph{$M$-roughly starlike} if there exists a base point $p \in X$ such that for every $x \in X$, there exists a geodesic ray $\gamma : [0, \infty) \to X$ with $\gamma(0)=p$ satisfying $\dist_{d}(x, \gamma) \leq M$.
\end{definition}

We define $\delta$-Gromov hyperbolic spaces and Gromov boundary.
\begin{definition}[Gromov product]
	Let $(X, d)$ be a metric space. The \emph{Gromov product} of two points $x, y \in X$ with respect to $p \in X$ is defined by
	\[
	 (x|y)_{p}:= \frac{1}{2}(d(p, x)+d(p, y)-d(x, y)).
	\]
\end{definition}

\begin{definition}[Gromov hyperbolic space]
Let $(X, d)$ be a metric space and $\delta \geq 0$. We say that $X$ is a \emph{$\delta$-Gromov hyperbolic space} if $X$ is unbounded, proper, geodesic metric space such that for all $x, y, z, p \in X$, 
\[
(x|z)_{p}\geq (x|y)_{p}\wedge(y|z)_{p}-\delta.
\]
\end{definition}

\begin{definition}[Gromov boundary]
	We say that two geodesic rays $\gamma$ and $\tilde{\gamma}$ in $X$ with $\gamma(0)=\tilde{\gamma}(0)=p$ are equivalent if $\sup_{t \geq 0}d(\gamma(t), \tilde{\gamma}(t))$ is finite. Then we can consider a quotient space of the set of all geodesic rays emanating from $p \in X$ by the above equivalence relation, denoted by $\partial_{G}X$.
\end{definition}
\begin{remark}\label{equivalence relation for Gromov sequences}
	Let $(X, d)$ be a $\delta$-Gromov hyperbolic space and $p \in X$ be a fixed point. There is another construction of Gromov boundary through Gromov sequences. Recall that we say that $ (x_n)_n \subseteq X$ is a Gromov sequence if $(x_n| x_m)_{p} \to \infty$ as $n, m \to \infty$. We consider the quotient space of all Gromov sequences where equivalence relation $\sim$ between two Gromov sequences $(x_n)_n$ and $(y_n)_n$ is given by $(x_n | y_n)_{p} \to \infty$ as $n \to \infty$. There is a canonical bijective map between the above two  Gromov boundaries, see \cite[Lemma~3.13]{BH}.
\end{remark}

We review a Harnack type inequality for the uniformization developed by Bonk-Heinonen-Koskela \cite{BHK}.  Let $(X, d) $ be a metric space and $p \in X$. For any $\epsilon>0$, a \emph{Harnack type inequality}
	\begin{equation}\label{Harnack}
		e^{-\epsilon d(x, y)} \leq \frac{e^{-\epsilon d(p,x)}}{e^{-\epsilon d(p,y)}} \leq e^{\epsilon d(x, y)}
	\end{equation} 
	holds for every $x, y \in X$, see \cite[Chapter 5]{BHK}. 
	
	We next recall the construction of the canonical boundary map $\Phi$.

\begin{definition}[Canonical boundary map $\Phi$]\label{Def of boundary map}
	Let $(X, d)$ be a $\delta$-Gromov hyperbolic space and $p \in X$ be a fixed point for uniformization with the parameter $\epsilon>0$. Let $\partial_{d_{\epsilon}} X^{\epsilon}:=\overline{X^{\epsilon}}\setminus X^{\epsilon}$ be the metric boundary of $X^{\epsilon}$. It is easy to see that for each geodesic ray $\gamma : [0, \infty) \to X$, $(\gamma(k))_k$ is a sequence converging to some point $x \in \partial_{d_{\epsilon}}X^{\epsilon}$ with respect to $d_{\epsilon}$. The limit point $x$ is independent of the choice the geodesic rays in the same equivalence class. Hence we define the map $\Phi : \partial_{G}X \to \partial_{d_{\epsilon}}X^{\epsilon}$ by $[\gamma] \mapsto \lim\limits_{k \to \infty}\gamma(k)$ where $[\gamma] \in \partial_{G}X$ is the equivalence class of a geodesic ray $\gamma$.
\end{definition}

We next review uniform spaces. 
\begin{definition}[$A$-uniform curve]\label{uniform curve}
	Let  $A>0$ and a metric space $(X, d)$ be given. Let $\partial_d X :=\overline{X}\setminus X$ be the metric boundary of $X$. For $a, b \in \mathbb{R}$ with $a<b$, we say that a curve $\gamma : [a, b] \to X$ is an \emph{$A$-uniform curve} if the curve $\gamma$ satisfies 
	\begin{enumerate}
		\item $l_{d}(\gamma) \leq A d(\gamma(a), \gamma(b))$,
		\item $l_{d}(\gamma|_{[a, t]})\wedge l_{d}(\gamma|_{[t, b]}) \leq A$ $\dist_{d}(\gamma(t), \partial_d X)$ \ \ \ for every $t \in [a, b]$.
	\end{enumerate}
\end{definition}

\begin{definition}($A$-uniform space)\label{uniform space}
	 A noncomplete locally compact metric space $(X, d)$ is called an \emph{$A$-uniform space} if every pair of points in $X$ can be connected by an $A$-uniform curve.
\end{definition}

\section{Proof of Theorem \ref{GH decomposition 1}}
In this section we prove Theorem \ref{GH decomposition 1}. Remark that the Gehring-Hayman theorem implies the other conditions in Theorem \ref{GH decomposition 1}. The rest of implications will be proved one by one.
\subsection{(1) $\rightarrow$ (2)}\label{equivalence of localized GH}
In this subsection we prove the Gehring-Hayman property for metric boundary points implies the Gehring-Hayman property for Gromov sequences and the bijectivity of the canonical boundary map $\Phi$. 
\begin{lemma}
	Let $(X, d)$ be a $\delta$-Gromov-hyperbolic space and $X^{\epsilon}$ be the uniformized space. If the Gehring-Hayman property for metric boundary points holds, then the Gehring-Hayman property for Gromov sequences holds.
\end{lemma}
\begin{proof}
	 Let $(x_n)_n$ be an arbitrary Gromov sequence. Then by the first inequality of \cite[Proof of Lemma~4.10]{BHK}, there exists $x \in \partial_{d_{\epsilon}}X^{\epsilon}$ such that $x_n \to x$ as $n \to \infty$ with respect to $d_{\epsilon}$. The conclusion follows by the Gehring-Hayman theorem for metric boundary points.
\end{proof}

\begin{lemma}
	Let $(X, d)$ be a $\delta$-Gromov-hyperbolic space and $X^{\epsilon}$ be the uniformized space with the distinguished point $p\in X$. Suppose that the Gehring-Hayman property for metric boundary points holds. Then the canonical boundary map $\Phi : \partial_{G}X \to \partial_{d_{\epsilon}}X^{\epsilon} $ is bijective. 
\end{lemma}
\begin{proof}
Take $x \in \partial_{d_{\epsilon}} X^{\epsilon}$ and geodesic rays $\gamma$ and $\tilde{\gamma}$ such that 
\[
\gamma(0)=\tilde{\gamma}(0)=p \ \text{and} \ \lim\limits_{n \to \infty}\gamma(n)=\lim\limits_{n \to \infty}\tilde{\gamma}(n)=x
\] 
with respect to $d_{\epsilon}$. The sequence $(z_n)_n:=(\gamma(1), \tilde{\gamma}(1), \gamma(2), \tilde{\gamma}(2), \cdots)$ is a  sequence converging to $x \in \partial_{d_{\epsilon}}X^{\epsilon}$. Thus by the Gehring-Hayman theorem for metric boundary points and  the second inequality of \cite[Proof of Lemma~4.10]{BHK}, we have $(z_n|z_m)_p \to \infty$ as $n, m \to \infty$, which implies $(\gamma(n))_n \sim (\tilde{\gamma}(n))_n$ as Gromov sequences. Hence the geodesic rays $\gamma$ and $\tilde{\gamma}$ are in the same equivalence class. Therefore, $\Phi$ is injective.

The surjectivity of $\Phi$ directly follows from the Gehring-Hayman property for metric boundary points and \cite[Lemma~4.10]{BHK}.


\end{proof}

\subsection{(2) $\rightarrow$ (3)}\label{localized GH to GH}
In this subsection we will prove that the second condition in Theorem \ref{GH decomposition 1} implies the Gehring-Hayman theorem. We first give a remark on Gromov sequences.

\begin{remark}\label{Gromov sequence remark}
	Let $(X, d)$ be a $\delta$-Gromov-hyperbolic space. For given $(x_n)_n$, $(y_n)_n \subseteq X$, suppose that $(y_n)_n$ is a Gromov sequence and $(x_n | y_n)_p \to \infty$ as $n \to \infty$. Then $(x_n)_n$ is a Gromov sequence which is equivalent to $ (y_n)_n$. In fact, since $X$ is a $\delta$-Gromov hyperbolic space,
	\begin{align}
		(x_n|x_m)_p &\geq (x_n | y_n)_p \wedge (y_n|x_m)_p-\delta \notag \\
		&\geq (x_n | y_n)_p \wedge \Big((y_n|y_m)_p\wedge (y_m|x_m)_p-\delta\Big)-\delta \notag \\
		& \to  \infty \ \ \ (n, m \to \infty),
	\end{align} 
	which tells us that $(x_n)_n$ is a Gromov sequence.
	\end{remark}
	
Next we derive some properties from the map $\Phi : \partial_{G}X \to \partial_{d_{\epsilon}} X^{\epsilon}$. Note that the only one thing we know is that $\Phi$ is a bijective map. For each $x \in \partial_{d_{\epsilon}} X^{\epsilon}$, take $(x_n)_n$ and $(y_n)_n$ such that $x_n \to x$ and $y_n \to x$ as $n \to \infty$ with respect to $d_{\epsilon}$. We then show that
\begin{itemize}
	\item We can extract Gromov sequences from $(x_n)_n$ and $(y_n)_n$.
	\item The extracted Gromov sequences are equivalent to each other.
\end{itemize} 
 The second property derived from the bijectivity of $\Phi$ plays an important role to prove Proposition~\ref{local to global}.

\begin{lemma}\label{Gromov subsequence}
	 Let  $x \in \partial_{d_{\epsilon}}X^{\epsilon}$. For each sequence $(x_n)_n \subseteq X$ converging to $x$ with respect to $d_{\epsilon}$, there always exists a Gromov subsequence $(x_{n_k})_k$.
\end{lemma}
\begin{proof}
	By taking a subsequence if needed, we may assume that $d(p, x_n)\geq n$ for each $n \in \mathbb{N}$. Take a geodesic curve $\gamma_n$ from $p$ to $x_n$. Applying Arzela-Ascoli theorem and doing a diagonal argument, there exist a subsequence $(\gamma_{n_k})_k$ and a geodesic ray $\gamma$ such that 
	\begin{equation}\label{estimate of Gromov subsequence}
		\lim\limits_{k \to \infty}\sup_{t \in [0, m]}d(\gamma_{n_k}(t), \gamma(t))=0  \ \ \text{and} \ \ \sup_{t \in [0, m]}d(\gamma_{n_m}(t), \gamma(t))\leq 1.
	\end{equation}
	for each $m \in \mathbb{N}$.
 Set $y_m=\gamma_{n_m}(m)$. Note that
 \begin{align}\label{Gromov seq 1}
 	(y_m|x_{n_m})_p &= \frac{1}{2}\Big(d(p, y_m)+d(p, x_{n_m})-d(y_m, x_{n_m})\Big)\notag \\
 	&=\frac{1}{2}\Big(m+d(p, x_{n_m})-(d(p, x_{n_m})-m)\Big)\notag \\
 	&\geq m \to \infty \ \ (m \to \infty).
 \end{align}
 Also, we have 
 \begin{align}\label{Gromov seq 2}
 	(y_m|\gamma(m))_p &=\frac{1}{2}\Big(d(p, y_m)+d(p, \gamma(m))-d(y_m, \gamma(m))\Big) \notag \\
 	&\geq \frac{1}{2}(m+m-1) \to \infty \ \ (m \to \infty).
 \end{align}
 where we used \eqref{estimate of Gromov subsequence} to obtain the last inequality. Therefore, combining Remark \ref{Gromov sequence remark}, \eqref{Gromov seq 1} and \eqref{Gromov seq 2}, we conclude that $(x_{n_m})_m$ is a Gromov sequence. This completes the proof.
 \end{proof}

\begin{lemma}\label{inverse is well defined}
	 Let $x \in \partial_{d_{\epsilon}} X^{\epsilon}$. Suppose that $\Phi :  \partial_{G} X \to \partial_{d_{\epsilon}}X^{\epsilon}$ is  bijective. If there are sequences $(x_n)_n$ and $(y_n)_n$ that converge to $x$ with respect to $d_{\epsilon}$, then any Gromov subsequences $ (x_{n_k})_k$ and $(y_{n_k})_k$ are equivalent to each other.
\end{lemma}
\begin{proof}
	Fix $x \in \partial_{d_{\epsilon}}X^{\epsilon}$. Let $(x_n)_n$ and $(y_n)_n$ be sequences converging to $x \in \partial_{d_{\epsilon}}X^{\epsilon}$. We will first prove that there exist Gromov subsequences $(x_{n_k})_k$ and $(y_{n_k})_k$ that are equivalent to each other as Gromov sequences. Applying the proof of Lemma \ref{Gromov subsequence} to $(x_n)_n$ and $(y_n)_n$, there exist Gromov subsequences $(x_{n_k})_k$ and $(y_{n_k})_k$ and geodesic rays $\gamma$ and $\tilde{\gamma}$ such that
	\begin{equation}
		(x_{n_k})_k \sim (\gamma(k))_k, \ \ (y_{n_k})_k \sim (\tilde{\gamma}(k))_k,
	\end{equation}
	as Gromov sequences and 
	\begin{equation}
		\lim\limits_{k \to \infty}\gamma(k)=\lim\limits_{k \to \infty}\tilde{\gamma}(k)=x
	\end{equation}
	with respect to $d_{\epsilon}$. Since $\Phi :  \partial_{G} X \to \partial_{d_{\epsilon}}X^{\epsilon}$ is injective, $\gamma$ and $\tilde{\gamma}$ are equivalent to each other, i.e., $(\gamma(k))_k \sim (\tilde{\gamma}(k))_k$ as Gromov sequences. Hence we conclude that $(x_{n_k})_k \sim (y_{n_k})_k$. By the above argument and noting that every Gromov sequence $(x_n)_n$ is equivalent to its subsequence $(x_{n_k})$ (\cite[Lemma~5.3]{Jussi}), we can prove that any Gromov subsequences $ (x_{n_k})_k$ and $(y_{n_k})_k$ are equivalent to each other.
\end{proof}

The following proposition tells us that the Gehring-Hayman property always holds for any $\epsilon>0$ if the distance between points is bounded above by a uniform constant. This will be used to prove Proposition \ref{local to global}. 

\begin{proposition}\label{local is fine.}
	Let $(X, d)$ be a geodesic metric space. Let $\epsilon>0$ and  $M >0$. Then there exists $C:=C(\epsilon, M) \geq 1$ such that for every pair of points $x, y \in X$ with $d(x, y) \leq M$,
	\[
	l_{d_{\epsilon}}([x, y]) \leq C d_{\epsilon}(x, y)
	\]
	holds.
\end{proposition}
\begin{proof}
Let $z\in [x,y]$, then $d(x,p)\le d(x,z)+d(z,p)\le M +d(z,p)$. Thus
\begin{align}\label{estimate 1}
l_{d_{\epsilon}}([x, y]) = \int_{[x, y]}\rho_{\epsilon}(s)\, ds \leq e^{-\epsilon (d(p, x)-M)}d(x, y).
\end{align}
	Let $\lambda$ be a rectifiable curve from $x$ to $y$. We examine two cases.\\
	\textbf{Case 1} : $\lambda \subseteq B_d(x, M)$\\
	 For $z \in \lambda$, we have $d(p, z) \leq d(p, x)+d(x, z) \leq d(p, x)+M$. Hence
	 \begin{align}\label{estimate 2}
		l_{d_{\epsilon}}(\lambda) \geq e^{-\epsilon(d(p, x)+M)}l_d(\lambda) \geq e^{-\epsilon(d(p, x)+M)}d(x, y).
	\end{align}
	\textbf{Case 2} : $\lambda \not\subseteq B_d(x, M)$\\
	In this case we have
	 \begin{align}\label{estimate 3}
	 	l_{d_{\epsilon}}(\lambda) &\geq \int_{\lambda \cap B_d(x, M)}  \rho_{\epsilon}\, ds \geq e^{-\epsilon (d(p, x)+M)} d(x, y)
	 \end{align}
	 Combining \eqref{estimate 1}, \eqref{estimate 2}, and \eqref{estimate 3}, we have 
	 \[
	 l_{d_{\epsilon}}([x, y]) \leq e^{2 \epsilon M}d_{\epsilon}(x, y).
	 \]
\end{proof}

Lastly, we prove the following.
\begin{proposition}\label{local to global}
	Let $(X, d)$ be a geodesic metric space and $X^{\epsilon}$ be the uniformized space. Suppose the Gehring-Hayman property for Gromov sequences holds and $\Phi : \partial_{G}X \to \partial_{d_{\epsilon}}X^{\epsilon}$ is bijective. Then the Gehring-Hayman theorem holds. 
\end{proposition}

\begin{proof}
	Suppose that the Gehring-Hayman theorem does not hold. Then there exist $(x_n)_n$, $(y_n)_n \subseteq X^{\epsilon}$ such that 
	\begin{equation}\label{GH proof estimate 1}
		l_{d_{\epsilon}}([x_n, y_n]) \geq n d_{\epsilon}(x_n, y_n).
	\end{equation}
	Note that by the proof of \cite[Lemma~4.10]{BHK}, we know that the LHS of \eqref{GH proof estimate 1} is bounded above by a uniform constant $C>0$ which only depends on $\epsilon$ and $\delta$. Dividing both sides of \eqref{GH proof estimate 1} by $n$ implies that 
	\begin{equation}\label{GH proof estimate 2}
	d_{\epsilon}(x_n, y_n) \leq C/n\to 0,
	\end{equation}
	as $n \to \infty$. We first claim that both $(x_n)_n$ and $(y_n)_n$ are unbounded. For this sake, it is enough to prove that either of these two sequences is unbounded. In fact, if one of these two sequences is unbounded and the other one is bounded, then  by taking a subsequence if needed, $d_{\epsilon}(x_n, y_n)$ is uniformly bounded from below by a positive constant, which contradicts \eqref{GH proof estimate 2}. If both $(x_n)_n$ and $(y_n)_n$ are bounded, then Proposition \ref{local is fine.} gives us a contradiction for a large enough $n \in \mathbb{N}$. Hence the claim follows. 
	
	By Arzela-Ascoli theorem, there exist a subsequence $(x_{n_k})_k$ and a geodesic ray $\gamma$ such that 
		$d_{\epsilon}(x_{n_k}, \gamma(\infty)) \to 0$ 
	as $k \to \infty$ where $\gamma(\infty) \in \partial_{d_{\epsilon}}X^{\epsilon}$ is the limit of the sequence $(\gamma(k))_k$ with respect to $d_{\epsilon}$. We note that $d_{\epsilon}(\gamma(\infty), y_{n_k}) \to 0$ as $k \to \infty$ since $d_{\epsilon}(x_{n_k}, y_{n_k}) \to 0$ as $k \to \infty$. By Lemma \ref{Gromov subsequence}, we can extract further subsequences $(x_{n_k})_k$ and $(y_{n_k})_k$ that are Gromov sequences.  By Lemma \ref{inverse is well defined},  $(x_{n_k})_k$ and $(y_{n_k})_k$ are equivalent to each other as Gromov sequences. By \cite[Lemma 5.3 (3)]{Jussi}, the sequence $(z_n)_n:=(x_{n_1}, y_{n_1}, x_{n_2}, y_{n_2, \cdots})$ is a Gromov sequence. Since the Gehring-Hayman property for Gromov sequences holds, there exists a constant $C\geq 1$ such that for all $k, l \in \mathbb{N}$,
	\[
	l_{d_{\epsilon}}([x_{n_k}, y_{n_l}]) \leq C d_{\epsilon} (x_{n_k}, y_{n_l}),
	\]
	which contradicts \eqref{GH proof estimate 1}. This completes the proof.
\end{proof}

\subsection{(4) $\rightarrow$ (3)}\label{From uniformity to GH}
In this subsection we prove that $X^{\epsilon}$ being uniform implies the Gehring-Hayman theorem. To do this, we first recall the quasihyperbolization of a uniform space.
\begin{definition}(Quasihyperbolization)
	Let $(\Omega, d)$ be an $A$-uniform space.  The \emph{quasihyperbolic metric $k$} is defined by
	\[
	k(x, y):= \inf\limits_{\gamma}\int^{l_d(\gamma)}_{0}\frac{1}{d(\gamma(t))}\,dt
	\]
	where the infimum is taken over all rectifiable curves $\gamma$ from $x$ to $y$ and $d(\cdot):= \dist_d( \cdot , \partial \Omega)$.  
\end{definition}
By \cite[Theorem 3.6]{BHK}, if $(\Omega, d)$ is an $A$-uniform space, then $(\Omega, k)$ is a proper geodesic $\delta$-Gromov hyperbolic space for some $\delta=\delta(A)$. Moreover, if $(\Omega, d)$ is bounded, then $(\Omega, k)$ is $M$-roughly starlike for some $M=M(A)$. 

\begin{definition}
	Let $(X, d)$ be a metric space and $C>0$. We say that a curve $\gamma : [a, b] \to X$ is a \emph{$C$-quasigeodesic} if 
	\begin{equation}
		\frac{1}{C}|t-t'|\leq d(\gamma(t), \gamma(t')) \leq C|t-t'|
	\end{equation}
	holds for all $t, t' \in [a, b]$.
\end{definition}

The following is essentially  \cite[Proposition~4.37]{BHK}. We remark that the restriction on the unifromization parameter $\epsilon$ is not needed to prove it.
\begin{proposition}
Let $(X,d)$ be an $M$-roughly starlike $\delta$-Gromov hyperbolic space. Then for $\epsilon>0$ there exists $C=C(\epsilon, M)$ such that for alla $x,y\in X$
\begin{equation}\label{key inequality0}
	\frac{1}{C}d(x,y)\le k_\epsilon(x,y)\le C d(x,y),
\end{equation}
where $k_\epsilon$ is the quasihyperbolic metric of $X^{\epsilon}$.
\end{proposition}
\begin{proof}
Let $\gamma$ be a rectifiable curve in $(X,d)$ arc-length parametrized by $d$. By Lemmas A.5 and A.7 in Appendix in \cite{BHK}  we know that $\gamma$ is also rectifiable in $(X,d_\epsilon)$ and moreover there exists a reparametrization $\gamma^o:[0,l_{d_\epsilon}(\gamma)]\to X$ of $\gamma$ with respect to $d_{\epsilon}$ such that $\gamma=\gamma^o\circ s_\epsilon$, where $s_\epsilon(t):=l_{d_\epsilon}(\gamma|_{0,t})=\int_0^t\rho_\epsilon(\gamma(t))\,dt$. Thus setting $d_{\epsilon}(x):=\dist_{d_{\epsilon}} (x, \partial_{d_{\epsilon}} X^{\epsilon})$, we have
\[
l_{k_\epsilon}(\gamma)=\int_0^{l_{d_\epsilon}(\gamma)}\frac{1}{d_\epsilon(\gamma^o(t))}\,dt=\int_0^{l_d(\gamma)}\frac{s_\epsilon'(t)}{d_\epsilon(\gamma^o\circ s_\epsilon(t))}\,dt= \int_0^{l_d(\gamma)}\frac{\rho_\epsilon(\gamma(t))}{d_\epsilon(\gamma(t))}\,dt,
\]
where the first equality comes from \cite[Lemma~A.7]{BHK} and we used the change of variables in the second equality. By \cite[Lemma 4.16]{BHK} there exists $C=C(\epsilon, M)\ge 1$ such that for every $x\in X$ we have
\begin{equation}\label{key inequality1}
	\frac{1}{C}\le \frac{\rho_\epsilon(x)}{d_\epsilon(x)}\le C.
\end{equation}
Combining the above we get
\begin{equation}\label{key inequality2}
	\frac{1}{C}l_d(\gamma)\le l_{k_\epsilon}(\gamma)\le C l_d(\gamma).
\end{equation}
Let $x,y\in X$ and let $\gamma$ be rectifiable curve joining them. The above gives that the first inequality of \eqref{key inequality0} by simply taking the infimum over all curves joining $x$ and $y$.  The second inequality of \eqref{key inequality0} is obtained by choosing $\gamma$ to be the geodesic between $x$ and $y$ with respect to $d$.
\end{proof}

The following lemma is originally stated for geodesics with respect to $k_{\epsilon}$ in \cite[Theorem~2.10]{BHK}. We note that the consequence of \cite[Theorem~2.10]{BHK} still holds true for quasigeodesics with respect to $k_{\epsilon}$ by the straightforward modification of their proof, see also \cite[Fact~2.10]{H}.
\begin{lemma}(\cite[Theorem~2.10]{BHK})\label{quasihyperbolization key}
	Let $(\Omega, d)$ be an $A$-uniform space and $k$ be the quasihyperbolic metric with respect to  $d$. Then every $C$-quasigeodesic with respect to $k_{\epsilon}$ is a $B$-uniform curve with respect to $d$ for some $B:=B(A, C)>0$. 
\end{lemma}

We are now in a position to prove that $X^{\epsilon}$ being uniform implies the Gehring-Hayman theorem.
\begin{proposition}\label{uniform to GH}
	Let $(X, d)$ be an $M$-roughly starlike $\delta$-Gromov hyperbolic space. If $X^{\epsilon}$ is an $A$-uniform space, the Gehring-Hayman theorem holds for $X^{\epsilon}$.
\end{proposition}
\begin{proof}
	
	By \eqref{key inequality2}, there exists $C>0$ such that for every geodesic $\gamma : [0, l_d(\gamma)] \to X$ with respect to $d$, we have
	\[
	\frac{1}{C}|t-t'|\leq l_{k_{\epsilon}}(\gamma) \leq C|t-t'|
	\]
	 for any $t, t' \in [0, l_d(\gamma)]$. This implies that geodesics with respect to $d$ are $C$-quasigeodesics with respect to $k_{\epsilon}$.  The conclusion follows by Lemma~\ref{quasihyperbolization key}.
\end{proof}
\begin{remark}
	By a careful observation of proofs in this subsection, one might notice that $\delta$-Gromov hyperbolicity does not play any role in proving Proposition \ref{uniform to GH}. This tells us that Proposition \ref{uniform to GH} holds as long as $(X, d)$ is $M$-roughly starlike and $X^{\epsilon}$ is uniform although we assume $\delta$-Gromov hyperbolicity, following convention. 
\end{remark}

\section{Critical exponents for some examples}\label{examples}
In this section we examine critical exponents for the uniformized space to be a uniform space in the case of the hyperbolic spaces, the model spaces $\mathbb{M}^{\kappa}_n$ of the sectional curvature $\kappa<0$ with the dimension $n \geq 2$ and hyperbolic fillings. 
\subsection{The hyperbolic spaces and the model spaces}
In this subsection we first show that the critical exponent for the uniformized space of the  hyperbolic spaces to be a uniform space is $\epsilon~=~1$. We first recall that for $n \geq 2$, the Poincar\'e ball model of the hyperbolic space is the unit open ball $\mathbb{B}^n \subseteq \mathbb{R}^n$ with the metric defined by 
\[
g_{\mathbb{B}^n}=\frac{4}{(1-(\sum_{i=1}^{n}x_i^2))^2}\sum_{i=1}^{n}dx_i^2,
\]
where $(x_1, \cdots, x_n) \in \mathbb{B}^n$. Note that in the case of $2$-dimensional hyperbolic space, one can see that in polar coordinates with respect to the hyperbolic metric we have
\begin{equation}\label{polar coordinate}
	g_{\mathbb{B}^2}=dr^2+\text{sinh}(r)^2d\theta^2.
\end{equation}
Let $d_{\mathbb{B}^n}$ be the Riemannian distance induced from the metric $g_{\mathbb{B}^n}$ on $\mathbb{B}^n$. We first uniformize the metric space $(X, d):=(\mathbb{B}^2, d_{\mathbb{B}^2})$ with the base point $p=~(0,0)$ and $\epsilon>1$. Let $\gamma$ and $\tilde{\gamma}$ be two geodesic rays with $\gamma(0)=~\tilde{\gamma}(0)=~(0, 0)$. From \eqref{polar coordinate} and $\epsilon > 1$, we have
\begin{align}
	d_{\epsilon}(\gamma(k), \tilde{\gamma}(k)) &\leq 2\pi e^{-\epsilon k}\text{sinh}(k) \to 0,
\end{align}
as $k \to \infty$. This implies that $\partial_{d_{\epsilon}}X^{\epsilon}$ is one point while $\partial_{G}X$ is the unit circle. Therefore by Theorem~\ref{GH decomposition 1}, $X^{\epsilon}$ is not a uniform space.  Moreover, since we can isometrically embed $(\mathbb{B}^2, d_{\mathbb{B}^2})$ into $(\mathbb{B}^n, d_{\mathbb{B}^n})$ for any $n \geq 3$, two geodesic rays $\gamma$ and $\tilde{\gamma}$ that are not equivalent in $(\mathbb{B}^2, d_{\mathbb{B}^2})$ can be seen as the ones in $(\mathbb{B}^n, d_{\mathbb{B}^n})$. The claim for the higher dimensional case follows by looking at those two geodesic rays $\gamma$ and $\tilde{\gamma}$ and the fact that $d_{\epsilon}(\gamma(k), \tilde{\gamma}(k)) \to 0$ as $k \to \infty$. On the other hand, Butler showed in \cite[Proposition~4.11 and Proposition~4.12]{B} that the Gehring-Hayman theorem holds for any CAT($-1$) space if  $0<\epsilon \leq 1$. Although his uniformization procedure is different from the one in \cite{BHK}, all the arguments to prove \cite[Proposition~4.11 and Proposition~4.12]{B} are valid for the uniformization developed by Bonk, Heinonen and Koskela. Hence we obtain the following corollary.
\begin{corollary}
	The critical exponent for the uniformized space of the hyperbolic space $(\mathbb{B}^n, d_{\mathbb{B}^n})$ to be a uniform space is $\epsilon=1$.
\end{corollary}

Let $\mathbb{M}^{\kappa}_n$ be the model space of the constant sectional curvature $\kappa<0$ with the dimension $n \geq 2$. The space $\mathbb{M}^{\kappa}_n$ is defined by $\mathbb{B}^n$ with the metric $g_{\mathbb{M}^{\kappa}_n}:=~(-1/\kappa) g_{\mathbb{B}^n}$. We  determine the critical exponent for the uniformized space of the model space $\mathbb{M}^{\kappa}_n$ to be a uniform space. Note that $d_{\mathbb{M}^{\kappa}_n}=(1/\sqrt{-\kappa})d_{\mathbb{B}^n}$. We first uniformize $(X, d):=(\mathbb{M}^{\kappa}_2, d_{\mathbb{M}^{\kappa}_2})$ with $p=(0, 0)$ and  $\epsilon > \sqrt{-\kappa}$. In the case of $\mathbb{M}^{\kappa}_2$, we have
	\begin{equation}\label{polar coordinate2}
		g_{\mathbb{M}^{\kappa}_2}=dr^2-\frac{1}{\kappa}\text{sinh}(\sqrt{-\kappa}r)^2d\theta^2
	\end{equation}
	 with respect to the polar coordinates. Hence  $X^{\epsilon}$ is not a uniform space by Theorem~\ref{GH decomposition 1} and looking at the circumference of a circle. The proof for the higher dimensional case follows as in the case of the hyperbolic spaces. We next examine the case where  $0<\epsilon \leq \sqrt{-\kappa}$. For this sake, we first prove the following simple lemma regarding the uniformization parameter for the Gehring-Hayman theorem under the scaling in a metric.
	 
	 \begin{lemma}\label{scaling metric}
	Let $(X, d)$ be a $\delta$-Gromov hyperbolic space. Let $K>0$ and $ \epsilon>0$. Set $(\tilde{X}, \tilde{d}):=(X, Kd)$ and $\tilde{\epsilon}:=\epsilon/K$. Then the Gehring-Hayman theorem holds for $X^{\epsilon}$ if and only if the Gehring-Hayman theorem holds for $\tilde{X}^{\tilde{\epsilon}}$. 
\end{lemma}
\begin{proof}
	It is enough to prove one implication. Assume that the Gehring-Hayman theorem holds for $(X, d_\epsilon)$. Let $\gamma$ be a rectifiable curve arc-length parametrized with respect to $d$ and $\tilde\gamma$ be the same curve but with the arc-length parametrization with respect to $\tilde{d}$. Notice that $l_{\tilde d}(\gamma)= kl_d(\gamma)$. Thus
\begin{align*}
		l_{d_{\epsilon}}(\gamma)&=\int_{0}^{l_d(\gamma)}e^{-\epsilon d( p, \gamma(t))}\, dt =\int_{0}^{l_d(\gamma)}e^{-\tilde{\epsilon} \tilde{d}( p, \gamma(t))}\, dt \\ &= \frac{1}{K}\int_0^{l_{\tilde d}(\gamma)}e^{-\tilde\epsilon\tilde d(p,\tilde\gamma(t))}\,dt=\frac{1}{K}l_{\tilde{d}_{\tilde{\epsilon}}}(\gamma),
\end{align*}
where the second last equality comes from the change of variable formula.
	Therefore, for $x,y\in X$ we have  
	\[
	l_{\tilde{d}_{\tilde{\epsilon}}}([x, y]) = Kl_{d_{\epsilon}}([x, y])\leq KC d_{\epsilon}(x, y)= C\tilde{d}_{\tilde{\epsilon}}(x, y),
	\]
	where the Gehring-Hayman theorem for $X^{\epsilon}$ was applied to obtain the above inequality. This completes the proof.
\end{proof}
Applying Lemma \ref{scaling metric} to $(X, d):=(\mathbb{B}^n, d_{\mathbb{B}^n})$, $0<\epsilon\leq 1$ and $K:=1/\sqrt{-\kappa}$, we conclude that the Gehring-Hayman theorem for the uniformized space of $(\mathbb{M}^{\kappa}_n, d_{\mathbb{M}^{\kappa}_n})$ holds if $0<\tilde{\epsilon}\leq \sqrt{-\kappa}$ by \cite[Proposition~4.11 and Proposition~4.12]{B}. Since the Gehring-Hayman theorem implies that the uniformized space is a uniform space, we have the following consequence.
	 \begin{corollary}
	 	Let $n \geq 2$ and $\kappa<0$. The critical exponent for the uniformized space of the model space $\mathbb{M}^{\kappa}_n$ to be a uniform space is $\epsilon = \sqrt{-\kappa}$.
	 \end{corollary}

\subsection{Hyperbolic fillings}
We briefly recall the construction and the properties of hyperbolic fillings, following \cite{BBS}. We refer readers to \cite[Section 3]{BBS} for more detailed explanations. Note that there are many slightly different definitions and variants of hyperbolic fillings appeared in  \cite{BSc}, \cite{BSa}, \cite{BP}, \cite{B} and \cite{K}. Before defining hyperbolic fillings, we recall some terms.  Given a metric space $(Z, d)$ and $r>0$, a set $E \subseteq Z$ is called an \emph{$r$-separated set} if $d(x, y)\geq r$ for every pair of points $x, y \in E$. The existence of a maximal $r$-separated set is ensured by Zorn's lemma. We say that $(Z, d)$ is \emph{precompact} if the completion of $Z$ is compact. Let $(Z, d)$ be a precompact metric space with $\diam(Z)<1$ and $\alpha, \tau>1$ be given parameters.  Take maximal $\alpha^{-n}$-separated sets $E_n$ with the property $E_n\subseteq E_{n+1}$. Define a vertex set by 
\[
V:= \bigcup V_n, 
\]
where $V_n:=\{(x, n) \ | \ x \in E_n \}$. For $(x, n), (y, m) \in V$, there is an edge between them if and only if either of the following is satisfied:
\begin{enumerate}
	\item $n=m$ and $B_d(x, \tau \alpha^n) \cap B_d(y, \tau \alpha^m) \neq \emptyset$.
	\item $n=m \pm 1$ and $B_d(x,  \alpha^n) \cap B_d(y, \alpha^m)\neq \emptyset$.
\end{enumerate}
 	The metric graph $V$ with edges defined in the above manner is called \emph{a hyperbolic filling of $Z$}, denoted by $X$. Note that a metric graph is a graph whose edges are identified with the unit interval $[0, 1]$. By \cite[Corollary~3.2, Theorem~3.4 and Proposition~4.6]{BBS}, $X$ is a $1/2$-roughly starlike $\delta$-Gromov hyperbolic space for some $\delta=\delta(\alpha, \tau)$. Moreover, it was shown in \cite[Theorem 5.1]{BBS} that the uniformized space $X^{\epsilon}$ of $X$ is a uniform space if $0<\epsilon \leq \text{log}(\alpha)$. We show that $X^{\epsilon}$ is not a uniform space if $Z$ is sufficiently nice and $\epsilon > \text{log}(\alpha)$. 
 	
 	\begin{corollary}\label{hyperbolic filling}
 		Let $\epsilon>\text{log}(\alpha)$. Assume a  precompact metric space $Z$ has at least two points and each pair of points in $Z$ is connected by a curve in $\overline{Z}$ whose length is bounded by some uniform constant. Then $X^{\epsilon}$ is not a uniform space. Therefore, the critical exponent for the uniformized space of the hyperbolic filling $X$ to be a uniform space is $\epsilon=\text{log}(\alpha)$.
 	\end{corollary}
 	\begin{proof}
 		By our assumption, we know that the metric boundary of the uniformized space of $X$ with the parameter $\epsilon > \text{log}(\alpha)$ consists of one point by \cite[Proposition 4.1]{BBS}. Since $Z$ has at least two points, there is no way to construct a bijective map between $\partial_{G}X$ and $\partial_{d_{\epsilon}}X^{\epsilon}$. Theorem~\ref{GH decomposition 1} tells us that the uniformized space $X^{\epsilon}$ of a hyperbolic filling $X$ is not a uniform space.  
 	\end{proof}
\begin{remark}
	It is well-known that metric trees are $0$-Gromov hyperbolic spaces. Although Corollary \ref{hyperbolic filling} tells us that the uniformized space of a hyperbolic filling is not a uniform space for $Z$ sufficiently nice and  $\epsilon > \text{log}(\alpha)$, the uniformized space of a metric tree is always a uniform space for every $\epsilon>0$. This is due to the fact that the Gehring-Hayman theorem always holds for any $\epsilon>0$. 
\end{remark}

}

\end{document}